\documentclass[11pt,a4paper]{amsart}
\usepackage{amssymb,amsmath,amsthm}
\usepackage{graphicx}
\usepackage{epsfig}

\long\def\onefigure#1#2{
\begin{figure*}[tbp]
\begin{center}
#1
\end{center}
\caption{#2}
\end{figure*}
} 

\newcommand{\lipefig}[2]  
{\onefigure{\mbox{\psfig{file=#1.eps}}}{\label{f:#1} #2} }

\theoremstyle{plain}
\newtheorem{theorem}{Theorem}[section]
\newtheorem{lemma}[theorem]{Lemma}
\newtheorem{prop}[theorem]{Proposition}

\newtheorem{claim}[theorem]{Claim}

\newcommand{\al}{\alpha}

\newcommand{\de}{\delta}

\newcommand{\eps}{\varepsilon}
\newcommand{\N}{\mathbb{N}}

\newcommand{\R}{\mathbb{R}}

\newcommand{\diam}{\textrm{diam\;}}
\newcommand{\conv}{\textrm{conv}}
\begin{document}

\title{Block partitions: an extended view}

\author{I. B{\'a}r{\'a}ny,  E. Cs\'oka, Gy. K\'arolyi, and G. T\'oth}

\begin{abstract} Given a sequence $S=(s_1,\dots,s_m) \in [0, 1]^m$, a block $B$ of $S$ is a subsequence $B=(s_i,s_{i+1},\dots,s_j)$.
The size $b$ of a block $B$ is the sum of its elements.
It is proved in \cite{BG} that for each positive integer $n$, there is a partition of $S$ into $n$
blocks $B_1, \dots , B_n$ with $|b_i - b_j| \le 1$ for every $i, j$.
In this paper, we consider a generalization of the problem in higher dimensions.
\end{abstract}

\maketitle

\section{Introduction}

This paper is a follow-up to \cite{BG}, which is about block partitions of sequences $S = (s_1, \dots, s_m)$ of real numbers.
A block $B$ of $S$ is either a sequence $B=(s_i, s_{i+1}, \dots, s_j)$  where $i\le j$ or the empty set.
The size $b$ of a block $B$ is the sum of its elements. One of the main
results of \cite{BG} says the following.

\begin{theorem} \label{th:orig} Given a sequence $S=(s_1,\dots,s_m)$ of
real numbers with $s_i \in [0,1]$ for all $i$, and an integer $n\in \N$,
there is a partition of $S$ into $n$ blocks such that $|b_i-b_j|\le 1$ for
all $i,j$.
\end{theorem}

The bound given here is best possible as shown by the example when all
the $s_i=1$ and $n$ does not divide $m$.

Here, we rephrase or generalize the setting and the result in the following way.
Define $a_i = \sum_{j=1}^i s_j$ (where $a_0 = 0$) and set $A = \{a_0,\ldots,a_m\}$.
A partition of $A$ into $n$ blocks is the same as choosing indices $x_0 = 0\le x_1 \le \cdots \le x_n = m$ so that $x_j$ is represented by $a_{x_j} \in A$, and then $b_j = a_{x_j} - a_{x_{j-1}}$ for all $i\in [n]$.
Here $[n]$ denotes the set $\{1, \ldots, n\}$.
For the more general setting, we consider closed sets $(A_i) = A_0, A_1, \ldots, A_n \subset \R$ satisfying the following conditions:
\begin{enumerate}
\item[(i)] $A_0 = \{0\}$, $A_n = \{s\}$.
\item[(ii)] $A_j \cap [a, a+1] \ne \emptyset$ for every $j \in [n-1]$
and for every $a \in \R$.
\end{enumerate}

A {\sl transversal} $T = (a_0, a_1, \dots, a_n)$ of the system $A_j$ is simply a selection of elements $a_j \in A_j$ for every $j = 0, 1, \dots, n$.
Given a transversal, we define $z_j = a_j - a_{j-1}$ for all $j \in [n]$.
In this setting, $z_j$ corresponds to the size $b_j$ of the $j$th block.

\begin{theorem}\label{th:transv}
Under the above conditions, there is a transversal $a_j \in A_j$ such that
\begin{equation*}
|z_i - z_j| \le 1 \mbox{ for all } i, j \in [n].
\end{equation*}
\end{theorem}

The bound $|z_i - z_j| \le 1$ is again best possible as shown by (essentially the same) example:
$A_0 = \{0\}$, $A_i = \{0, \pm1, \pm 2, \ldots\}$ for $i \in [n-1]$ and $A_n = \{s\}$, with an integer $s$ not divisible by $n$.
Without the closedness of the sets $A_j$, we would only have a transversal with $|z_i - z_j| \le 1 + \eps$ for each $\eps > 0$,
as shown by the following example: let $n \ge 3$, $A_j=(\infty,-1/2]\cup (1/2,\infty)$ for odd $j$,
$A_j=(\infty,-1/2)\cup [1/2,\infty)$ for even $j$.

Theorem~\ref{th:transv} could be easily deduced from Theorem~\ref{th:orig}, but we will present a new and shorter direct proof in the next section. Then we extend the new setting to higher dimensions.

Let $B_d$ be the unit ball of a norm $\|\cdot \|$ in $\R^d$.
Let $A_0, A_1, \ldots, A_n$ be a sequence of closed sets in $\R^d$.
It is called a {\sl grid-like sequence} if
(i) $A_0 = \{0\}$ and $A_n = \{s \}$ for a fixed element $s \in \R^d$, and
(ii) each of $A_1, A_2, \dots, A_{n-1}$ intersect with all unit balls, or formally, $\forall i \in [n-1],\ \forall a \in \R^d\colon$
\begin{equation}\label{eq:dense}
(a + B_d) \cap A_i \ne \emptyset.
\end{equation}
Note that in the one dimensional case we required
$A_i \cap [a,a+1]\ne \emptyset$ while the above condition would translate
to $A_i \cap [a-1,a+1]\ne \emptyset$. So there is a factor of 2 in the
new setting.

Given a transversal $T$, we define again $z_i = a_i - a_{i-1}$ for $i \in[n]$ and set $Z=\conv \{z_1,\ldots,z_n\}$.
The goal is to find a transversal $T$ such that
\begin{equation}\label{eq:target}
D(T) = D\big(T, (A_i)\big) = \diam Z = \max_{i,j} \|z_i - z_j\|
\end{equation}
is as small as possible. Let
\begin{equation*}
  D_d^* = \sup_{(A_i)} \min_{T} D\big(T, (A_i)\big)
\end{equation*}
for grid-like sequences $(A_i)$ and transversals $T$. Due to the closedness of each $A_i$, this minimum always exists. It is easy to see the following two propositions.
\begin{prop}
  $D^*_{d+1} \ge D^*_d$
\end{prop}
\begin{proof}
  If $A_0, A_1, ..., A_n \subset \R^d$ is a grid-like sequence, then $A_0, A_1 \times \R, A_2 \times \R ... A_{n-1} \times \R, A_n \subset \R^{d+1}$ is also a grid-like sequence, and the $n$-dimensional projection of any transversal of the latter sequence is a transversal of the former sequence with at most the same diameter.
\end{proof}

\begin{prop} \label{pr:triv}
$D^*_d \le 4$, or in other words, there is always a transversal $T$ with $D(T)\le 4$.
\end{prop}

\begin{proof}
Set $t=\frac sn$ and choose a point $a_i\in A_i$ from $it +B$ for $i\in [n-1]$, so $a_i=it +b_i$ with $b_i \in B$, for all $i=0,\ldots,n$. Then $z_i=t+b_i-b_{i-1}$, and so $z_i-z_j=b_i-b_{i-1}-b_j+b_{j-1}$, clearly in $4B$.
\end{proof}

\medskip
One could hope that the better bound $D^*_d \le 2$, which is valid for $d = 1$, also holds in higher dimensions. But this is not the case, at least with Euclidean norm:

\begin{theorem} \label{th:2.07}
\begin{equation*}
D^*_2 \ge 4 \sqrt{2 - \sqrt{3}} \approx 2.071.
\end{equation*}
In more detail, for every $\varepsilon > 0$ and $n \ge n_0(\varepsilon)$, there exists a grid-like sequence
$A_0, A_1, \ldots, A_n \subset \R^2$ such that for any transversal $T$, $D(T) \le 4 \sqrt{2 - \sqrt{3}}$.
\end{theorem}

We have a stronger bound $1 + \sqrt 2 \approx 2.414$, see Theorem~\ref{th:2.41} below, which may be sharp for $d = 2$ or maybe even for all $d$. Its proof is based on the same ideas as that of Theorem~\ref{th:2.07}, but it is much longer and more complicated case analysis. Therefore, we prove Theorem~\ref{th:2.07} and give an informal description of the construction for Theorem~\ref{th:2.41}, but omit the proof.

\medskip

\begin{theorem}\label{th:2.41}
$D^*_2 \ge 1 + \sqrt 2 \approx 2.414$.
\end{theorem}

Apart from the trivial bound $D^*_d \le 4$ given in Proposition \ref{pr:triv}, we cannot prove any upper bound,\footnote{Update: Endre Cs\'oka recently claimed an unpublished upper bound $2\sqrt{2}$ using topology.} not even in the case of the maximum norm $\|\cdot \|_\infty$. Note that the existence of a transversal $T$ with ``$D(T)\le 2$ in the maximum norm'' would imply the bound $D^*_d \le 2\sqrt{2}$ (in the Euclidean norm). For some related problems and results we refer to \cite{LPS}.

\medskip
\noindent{\bf Remark.}
We may assume without any loss that $s=0$. Indeed, with the
previous meaning of $t$, set $A_i^*=A_i-it$. Then the system $A_i^*$ with
$s^*=0$ satisfies condition (\ref{eq:dense}) and it is easy to  check that
for the transversals $a_i \in A_i$ and $a_i^*=a_i-it \in A_i^*$ one has the
same $z_i-z_j=z_i^*-z_j^*$ and then the diameters of $Z$ and $Z^*$ coincide.
Note that $0=\frac{1}{n} \sum_{i=1}^nz^*_i\in Z^*$.
\medskip

\section{Proof of Theorem \ref{th:transv}}

The idea is to find an $x \in \R$ such that the transversal we look for
satisfies the condition
\begin{equation}\label{eq:iter}
a_i \in a_{i-1}+[x,x+1].
\end{equation}

With this strategy the condition $|z_i-z_j|\le 1$ is automatically
guaranteed. The question is whether there is an $x\in \R$, which admits a
transversal $a_0,\dots,a_n$ satisfying condition (\ref{eq:iter})
for every $i\in [n]$.

We analyze what happens when this strategy is followed. A
{\sl partial transversal} is just a selection of $a_i \in A_i$
for $i=0,\ldots,h$, $h \in[n]$. We call it {\sl $x$-good} if
it satisfies condition (\ref{eq:iter}) for every $i\in [h]$.
As a first step, $a_1$ is to be chosen from
the interval $J_1=a_0+[x,x+1]=[x,x+1]$, and $J_1\cap A_1$ is
nonempty because of condition (ii). Thus, $a_0,a_1$ is an $x$-good
partial transversal for any $a_1 \in J_1\cap A_1$.

It is easy to see that $x$-{good} transversals exist for every $x \in \R$
and $h\in [n-1]$. To construct such partial transversals we define $J_i$
recursively as follows. Given $J_i$ for some $i\in [n-1]$ and a
fixed $x \in \R$, we let
\[
J_{i+1}=(J_i\cap A_i)+[x,x+1]=\bigcup\{[a+x,a+x+1]: a \in J_i\cap A_i\}.
\]

A routine induction, based on condition (ii), shows that $J_i$
is a closed interval of length at least one for every $i \in [n]$,
and so $J_i$ intersects $A_i$ if $i\in [n-1]$.
The intervals $J_i$ of course depend on $x$, and we write
$J_i(x)=[L_i(x),R_i(x)]$ to express this dependence.
The definition of $J_i(x)$ implies that
\begin{eqnarray}\label{eq:LR}
L_{i+1}(x)&=&x+\min\{a\in A_{i}: a\ge L_{i}(x)\},\\
R_{i+1}(x)&=&x+1+\max\{a\in A_{i}: a\le R_{i}(x)\}.
\end{eqnarray}

Note that both $L_i(x)$ and $R_i(x)$ are increasing functions of $x$
satisfying $R_i(x)-L_i(x)\ge 1$. Also, $L_{i+1}(x) \ge x+L_i(x)$ implying
via an easy induction that $L_i(x)$ and then $R_i(x)$ tends to infinity
as $x \to \infty$. A similar argument shows that
$\lim R_i(x)=\lim L_i(x)=-\infty$ as $x\to -\infty$. It is also clear that if
$a_h \in J_h(x)\cap A_h$ for some $x \in \R$, then there exists an
$x$-{good} partial transversal $a_0,a_1,\ldots,a_h$.

To complete the proof of Theorem~\ref{th:transv} we only have to show that
there is an $x \in \R$ such that $J_n(x)\cap A_n$ is non-empty, that is,
$s \in J_n(x)$.  Actually  we prove more:

\begin{lemma}\label{l:cover} For every $h=1,\ldots,n$ the intervals
$J_h(x)$ ($x \in \R$) cover every point of $\R$: $\bigcup_{x \in \R}J_h(x)=\R$.
\end{lemma}

\noindent{\bf Proof.}
First we claim that $L_i$ is a left continuous function for every $i\in [n]$.
This is evident for $i=1$, so we assume that $L_i$ is left continuous and
proceed to prove that $L_{i+1}$ is also left continuous. In view of
(\ref{eq:LR}) it amounts to checking that
\[
f(x)=\min\{a\in A_{i}: a\ge L_{i}(x)\}
\]
is a left continuous function of $x$. Let $x\in \R$ be arbitrary; we have to
show that $x_j<x$, $x_j\to x$ implies $f(x_j)\to f(x)$.
It is clear that if $[L_i(u),L_i(x))\cap A_i=\emptyset$ for some $u<x$,
then $f$ is constant on the interval $[u,x]$ and $f(x_j)=f(x)$ for $x_j\ge u$.

Otherwise there is a sequence $a_j\in A_i$ such that
\[
L_i(x_j)\le f(x_j)=a_j<L_i(x).
\]
Here $L_i(x_j)\to L_i(x)$ by the left continuity of $L_i$. Accordingly,
\[
L_i(x)=\lim_{j\to \infty} a_j\in A_i
\]
because $A_i$ is closed. It follows that $f(x_j)\to L_i(x)=f(x)$. The left
continuity of $f$ at $x$ is thus established.

Similarly, $R_i$ is a right continuous function for every $i\in [n]$.
To complete the proof of the lemma, consider any point $p \in \R$ and
define
\begin{eqnarray*}
L^p=L^p_h&=& \{x \in \R: L_h(x)\le p\},\\
R_p=R^p_h&=&\{x \in \R: R_h(x)\ge p\}.
\end{eqnarray*}
Here neither $L^p$ nor $R^p$ is empty since $\lim_{x\to -\infty}L_h(x)=-\infty$ and $\lim_{x\to \infty}R_h(x)=\infty$.
The continuity and monotonicity properties of the functions $L_h$ and $R_h$
imply that both $L^p$ and $R^p$ are
closed sets. Further, $L^p\cup R^p=\R$, as $x \notin L^p$ implies
$L_h(x) >p$ and $x \notin R^p$ implies $R_h(x) < p$, and then
$R_h(x)<p<L_h(x)$, which is impossible.

Now if two closed sets cover the connected space $\R$,
then they have a point in common.
So there is an $x_0 \in L^p\cap R^p$, and then $L_h(x_0)\le p\le
R_h(x_0)$, so $p\in J_h(x_0)$. Thus every $p\in \R$ is contained in
some $J_h(x)$.
\qed

\medskip
\noindent{\bf Remark.}
The proof of Theorem~\ref{th:orig}, which is an algorithm of
complexity $O(nm^3)$ can be modified to give another (algorithmic) proof of
Theorem~\ref{th:transv}. But the above proof can be turned into an
approximation algorithm the following way. By the remark at the end of
the introduction we may assume that $s=0$. Note that for $x=-1$ no
interval $J_h(x)$ contains a positive number, and similarly, for $x=0$ no
interval $J_h(x)$ contains a negative one. Using binary search, after $k$
iteration, one finds an $x \in [-1,0]$ that is within distance $2^{-k}$ of
the solution.

\section{Proof of Theorem \ref{th:2.07}}

We begin with an informal description of the construction. We fix a large
enough integer $m$. The grid-like sequence consists of $2m+2$ sets
$A_0,\ldots,A_{2m+1}$, where $A_0=A_{2m+1}=\{ 0 \}$.
Recall that in this case $0\in Z$. Each other $A_i$ is the
union of sets $N_i,S_i,E_i,W_i$ (corresponding to North, South, East, and
West) plus four corners $Q_i^{NE},Q_i^{SE},Q_i^{NW},Q_i^{SW}$, see the figures
below. The sets $A_i$ are symmetric about the $x$ and $y$ axes. The
sets $A_1,\ldots,A_m$ make up the first part of the construction.
For $i\in \{m+1,\ldots,2m\}$, $A_i$ is the refection of $A_{2m+1-i}$
about the line $x=y$. In particular, $A_m=A_{m+1}$.

The main characters in our construction are a square $X$ and segments $G_i$ of
changing length that are either horizontal or vertical. The Minkowski sum
$X + G_i$ is a hexagon shown on Figures 1 and 2. Its vertical and horizontal
sides (or vertices) are drawn with heavy lines, its oblique sides with thin
segments. The  horizontal and vertical sides (or vertices) are extended to
the sets $N_i,S_i,E_i,W_i$.

The main step of the proof is to show that an optimal tranversal
$a_0,\ldots,a_{2m+1}$ has no point in the corners and that it does not visit
the same region ($N$ or $S$ or $E$ or $W$) twice. More precisely, if
$a_i\in N_i$ and $a_j\in N_i$ for some $i<j$, then $a_h\in N_h$ for all
$h\in \{i,\ldots,j\}$, and the same for the components of type $S,E,W$.

\medskip
Formally we define the various components of $A_i$ for $1\le i\le m$ as
\begin{eqnarray*}
N(a,b) & = & \{(x,y): -a\le x \le a,\; y\ge b\}, \\
S(a,b) & = & \{(x,y): -a\le x \le a,\; y\le b\},\\
E(c,d) & = & \{(x,y): x\ge c, \; -d\le y\le d\}, \\
W(c,d)  & = & \{(x,y): x\le c, \; -d\le y\le d\},
\end{eqnarray*}
and
\begin{eqnarray*}
Q^{NW}(e,f) & = & \{(x,y): e\ge x,\; f\le y\}, \\
Q^{NE}(e,f) & = & \{(x,y): e\le x,\; f\le y\}, \\
Q^{SW}(e,f) & = & \{(x,y): e\ge x,\; f\ge y\}, \\
Q^{SE}(e,f) & = & \{(x,y): e\le x,\; f\ge y\}
\end{eqnarray*}
with appropriately chosen parameters $a,b,c,d,e,f$ depending on $i$.
We call $V(a,b) = N(a,b)\cup S(a,-b)$
and $H(c,d)  =  E(c,d)\cup W(-c,d)$
the vertical and horizontal parts, and the set
\[
Q(e,f) = Q^{NW}(-e,f)\cup Q^{NE}(e,f)\cup Q^{SW}(-e,-f)\cup Q^{SE}(e,-f)
\]
the union of the corners.
Set first $A_1=V_1\cup H_1\cup Q_1$, where
\[
V_1=V(3,1),\quad H_1=H(4,0),\quad Q_1=Q(4.5,1.5).
\]
Here the values 4.5 and 1.5 are chosen so that $A_1$, and later all other
$A_i$ satisfy condition (ii). Next, for $i=2,\ldots,m$ we define
$A_i=V_i\cup H_i\cup Q_i$, where
\[
V_i=V(3-(i-1)\delta,1),\ H_i=H(4-(i-1)\delta,0),\
Q_i=Q(4.5-(i-1)\delta,1.5)
\]
with $\delta=3/(m-1)$.  Note that $A_m$ is made up of four halflines and the
four corners. For $i=m+1,\ldots,2m$ we let $A_i$ be the reflected copy of
$A_{2m+1-i}$ about the line $x=y$.
It is easy to check that $A_0, A_1, \ldots , A_{2m+1}$ is a grid-like sequence
for $n=2m+1$.

\begin{figure}
\centering
\includegraphics[scale=0.8,, trim={0 3cm  0 0}, clip=true]{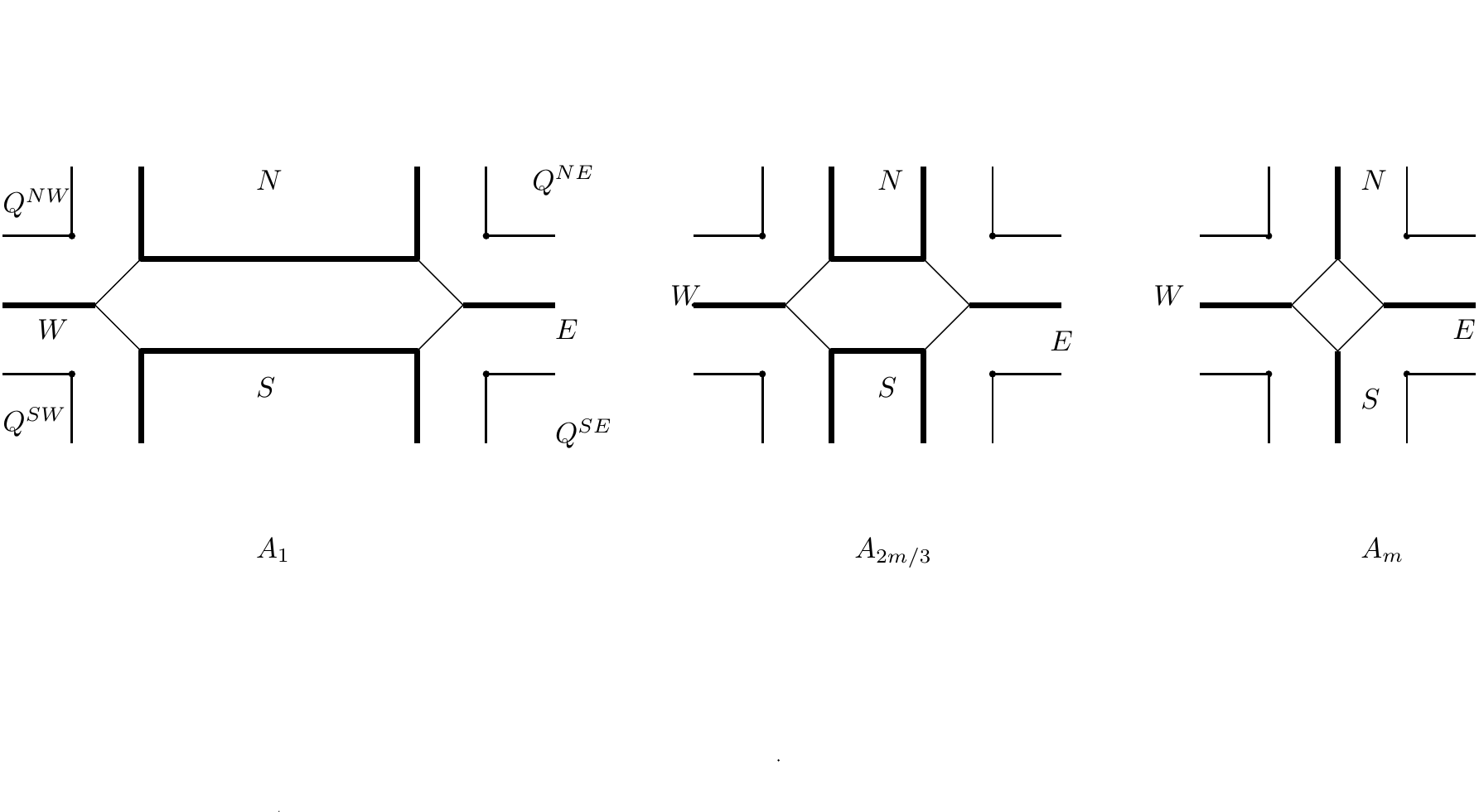}
\caption{Some $A_i$ from phase 1}
\label{fig:P1}
\end{figure}

\begin{figure}
\centering
\includegraphics[scale=0.8]{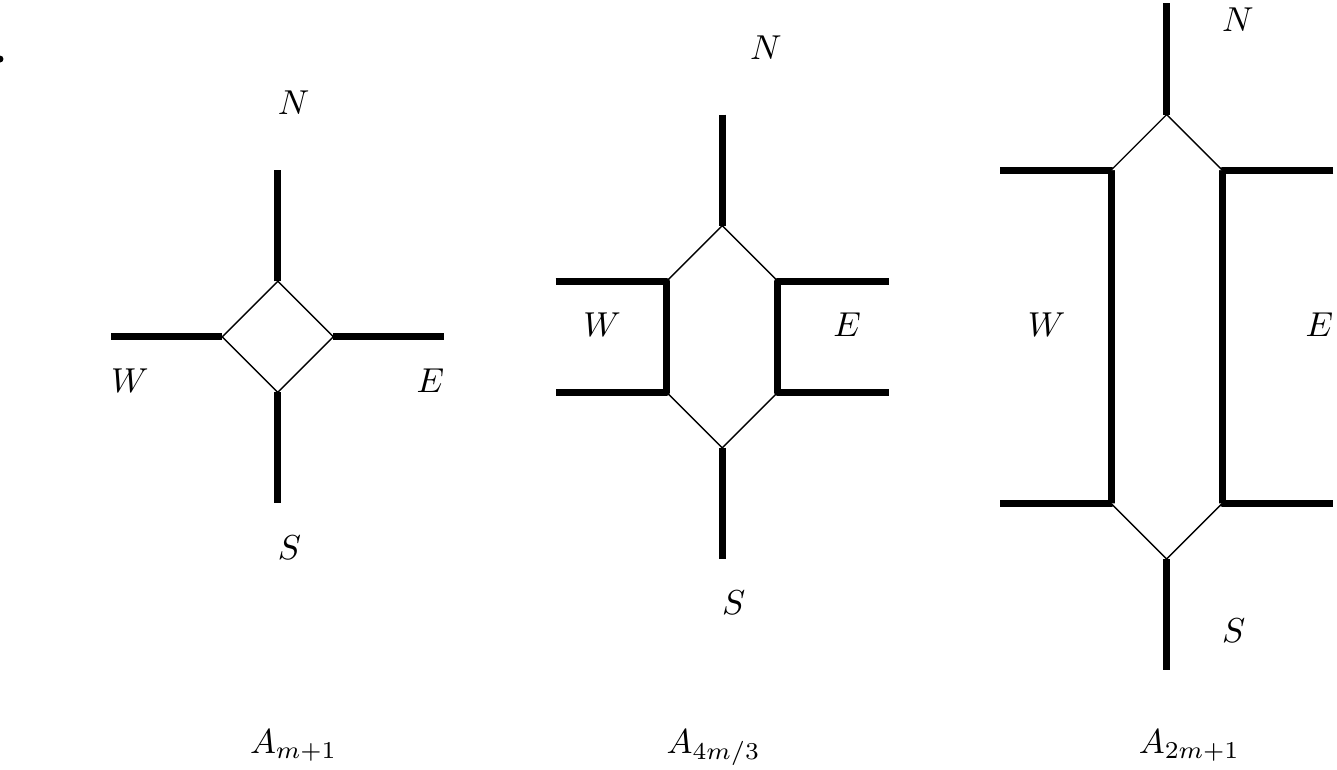}
\caption{Some $A_i$ from phase 2, corners suppressed}
\label{fig:P2}
\end{figure}

\begin{claim}
There is a transversal $T$ with $D(T) \le 4 \sqrt{2 - \sqrt 3}$.
\label{cl:opt}
\end{claim}

\noindent{\bf Proof.}
Consider the transversal $T=\{ a_0, \ldots , a_{2m+1}\}$ where
\[
a_i=\left( \frac{m-i}{m-1}\left(2\sqrt{3}-3\right),1\right), \quad
a_{m+i}=\left(1, \frac{i-1}{m-1}\left(2\sqrt{3}-3\right)\right)
\]
for $i=1,\dots,m$. Then $Z$ is an equilateral triangle whose vertices are
\[
z_1=\left(2\sqrt{3}-3,1\right),\quad z_{m+1}=(1,-1),\quad
z_{2m+1}=\left(-1,3-2\sqrt{3}\right),
\]
and $\diam Z=4\sqrt{2-\sqrt 3}$.
\qed

\medskip
Fix an optimal transversal $T=\{ a_0, \ldots , a_{2m+1}\}$,
such a transversal exists by compactness. Write $z_i=(x_i,y_i)$.
As $0\in Z$, the above claim implies $||z_i||\le 4\sqrt{2-\sqrt 3}$.

We say that $a_j$ {\sl jumps} if $a_{j-1}$ is in one type of
component in $A_{j-1}$ but $a_j$ is in another type in $A_j$.
Then $z_j$ is called the corresponting {\sl jump}. For instance
$a_j$ jumps if $a_{j-1} \in Q_{j-1}^{NW}$ but $a_j$ is in $N_j$
or in $W_j$.
The important property is that $||z_j||$ is large when $z_j$ is a
jump. Therefore the structure of the sequence of jumps is rather restricted.

Assume by symmetry that $a_1\in N_1$. Then $z_1=a_1$
and $y_1\ge 1$.
Similarly, we may assume that $a_{2m}\in E_{2m}$.
Then $z_n=-a_{n-1}$ and $x_n\le -1$.
\smallskip

{\bf Fact 1.} For every $i=1,\ldots,2m+1$, $y_i\ge -1.1$ and $x_i\le 1.1$.
Proof: Otherwise
the $y$ component of $z_1-z_i$ or the $x$-component of $z_i-z_n$
is larger than $2.1$ and then $\diam Z > 2.1$, which contradicts
Claim \ref{cl:opt} and the optimality of $T$.
\qed

\smallskip
This implies that the jumps $Q^{NW} \to N,\; N \to Q^{NE},\; Q^{SW} \to S,\;
S \to Q^{SE}$ and $W \to E$ are forbidden, and so are the jumps
$Q^{NW} \to W,\; W \to Q^{SW},\; Q^{NE} \to E,\;E \to Q^{SE}$ and $N \to S$,
see Figure~\ref{fig:forbid}. In particular, $Q^{SE}$ and $Q^{NW}$ is never
visited by $T$, because all other components $T$ could jump from $Q^{NW}$
into or from which $T$ could jump into $Q^{SE}$ are too far away.
Moreover, as indicated  on Figure~\ref{fig:forbid}, there cannot be
$E \to W$ or $S \to N$ jumps either. Indeed, suppose that there is a jump
from some
$E_i$ to $W_{i+1}$. To get back to the point $a_{2m}\in E_{2m}$ there must be
another jump to $E_j$ for some $j>i+1$. But then $x_{i+1}\le -2$ and $x_j\ge
1-\delta$, yielding $||z_j-z_{i+1}||\ge 3-\delta$, a contradiction. A similar
argument applies to an $S \to W$ jump.
\begin{figure}
\centering
\includegraphics[scale=0.8]{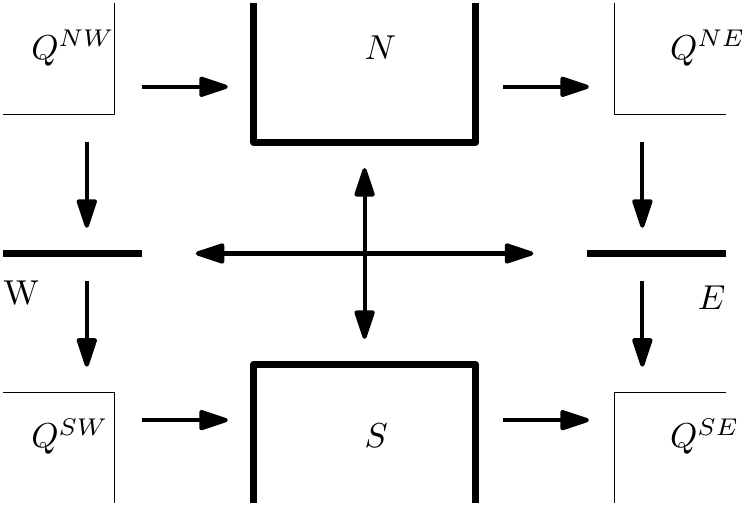}
\caption{Forbidden jumps}
\label{fig:forbid}
\end{figure}

\medskip

We say that a pair of jumps is {\sl  opposite} if either (a) one is $N \to E$ or $W \to S$ and the other is $E \to N$ or $S \to W$, or (b) one is $N \to W$ or $E \to S$ and the other is $W \to N$ or $S \to E$.

\smallskip

{\bf Fact 2.} There cannot be an  opposite pair of jumps if $m$ is large enough.
Proof: If $z_i$ and $z_j$ are  opposite jumps, then $|x_j-x_i|\ge
2-2\delta$ and $|y_j-y_i|\ge 2-2\delta$. Thus $||z_j-z_i||>2.8$ if $\delta$
is small enough, a contradiction.\qed

\begin{lemma}\label{l:nojump} There is a single jump and it goes
from $N_{j-1}$ to $E_j$ for some $j=2,\ldots,2m$.
\end{lemma}

\noindent
{\bf Proof.} There must be a jump since $a_1 \in N_1$ and $a_{2m}\in E_{2m}$.
Assume that the first jump is $z_j$, that is $a_i \in N_i$ for
$i=1,\ldots,j-1$ but $a_j \notin N_j$. As we have seen, $a_j$ is either in
$E_j$ or in $W_j$, $Q_j^{SW}$ and $Q_j^{SE}$ too far away.

Suppose first that $a_j\in W_j$, we will see that it leads to a contradiction.
Note that $z_j$ is fairly large: $x_j\le -1+\delta$ and $y_j \le -1+\delta$.
Also, there must be a further jump, say $z_k$ meaning that $a_i \in W_i$ for
$i=j,\ldots,k-1$ but $a_k \notin W_k$.
So $a_k$ is in $S_k$ or in $N_k$ since $Q_k^{NE}$
is too far away if $m$ is large enough. Now $a_{2m} \in E_{2m}$ which
cannot be reached from $a_k$ without creating an  opposite pair:
along the way there is an $S \to E$ jump or a $N \to E$ one. A contradiction.
\smallskip

We can conclude that the first jump goes from $N_{j-1}$ to $E_j$.

\smallskip

If there is a further jump, then the next jump, $z_k$ say,
must go from $E_{k-1}$ to $Q_k^{NE}$ or $S_k$ or $N_k$, as $Q_k^{SW}$ is
too far away. Here $N_k$ is excluded as then $z_j,z_k$ are  opposite.
If $a_k \in Q_k^{NE}$, then $||z_k-z_j||> |y_k-y_j|\ge 2.5-\delta$.
Assume finally
that $a_k \in S_k$. Then, again, $a_{2m} \in E_{2m}$ cannot be reached
from $a_k \in S_k$ without using an  opposite pair: along the way there
is an $S \to E$ jump or an $W \to N$ one. \qed

\medskip
Let $z_j$ be the single jump along the way.
Note that $x_j\ge 1-\delta$ and $y_j \le -1+\delta$.
This leads to a simple minimisation problem.
Given points $z_1,z_j,z_n \in \R^2$, with
$y_1\ge 1$, $x_j \ge 1-\delta$, $y_j \le -1+\delta$ and $x_n\le -1$,
find the minimum diameter of the triangle formed by these points.
As $\delta$ tends to 0, the unique solution to this problem converges to the
equilateral triangle specified in the proof of Claim \ref{cl:opt}.
Thus, $D(T) > 4 \sqrt{2 - \sqrt{3}} - \varepsilon$ for $m \ge m_0(\varepsilon)$.
This completes the proof of Theorem \ref{th:2.07} for odd values of $n$.
For even values of $n$ some straightforward modifications are needed, which are
left to the reader.

\section{Sketch of the proof of Theorem \ref{th:2.41}}

In this construction  $m$ is a large integer again, and the grid-like
sequence of sets is $C_{-1},C_0,\ldots$, $C_{6m+1}$, $C_{6m+2}$, where
$C_{-1}=C_{6m+2}=\{(0,0)\}$. Further $C_0=A_1,C_{6m+1}=A_{2m}$, where
$A_1,A_{2m}$ come from the previous construction to guarantee that
$|y_0|\ge1$ and $|x_{6m+1}|\ge 1$ with the previous notation
$z_i=a_i-a_{i-1}=(x_i,y_i)$.

The main characters are a rhombus $R$ (instead of the square $X$) and the
same segments $G_i$. The length of the shorter diagonal of the rhombus is
$\sqrt 2$, its sides are of length $\sqrt{2+\sqrt 2}$, and its  centre is
the origin. But this time rotated copies of $R$ are needed, so let $R(\al)$
stand for the rotated copy of $R$; rotated by angle $\al$ in clockwise
direction around its centre.

\begin{figure}
\centering
\includegraphics[scale=0.8, trim={0cm 3cm  0 0}, clip=true]{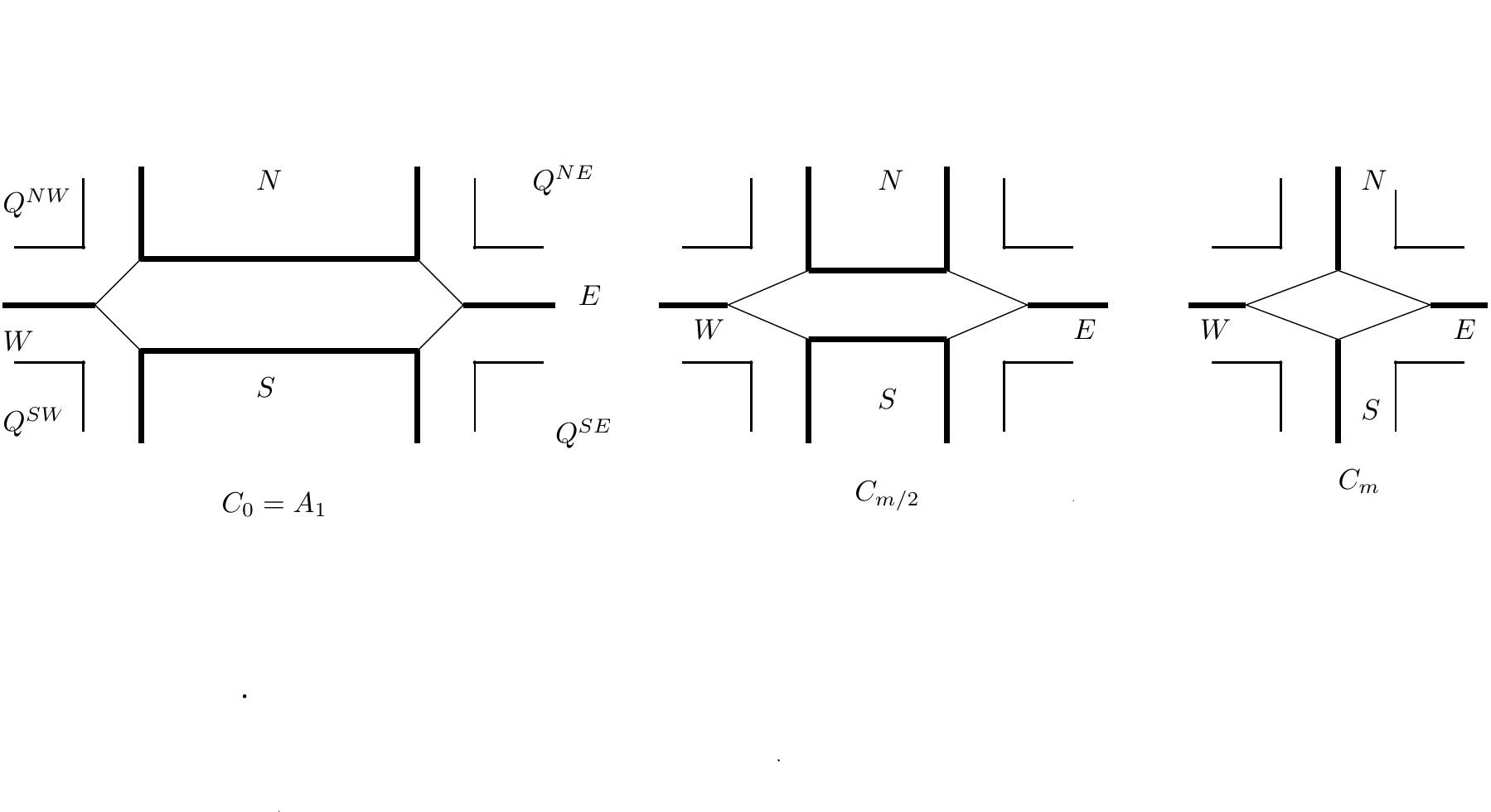}
\caption{$C_0,C_{m/2},C_m$, the rhombus $R$ is drawn with thin segments}
\label{fig:p1}
\end{figure}

\begin{figure}
\centering
\includegraphics[scale=0.8]{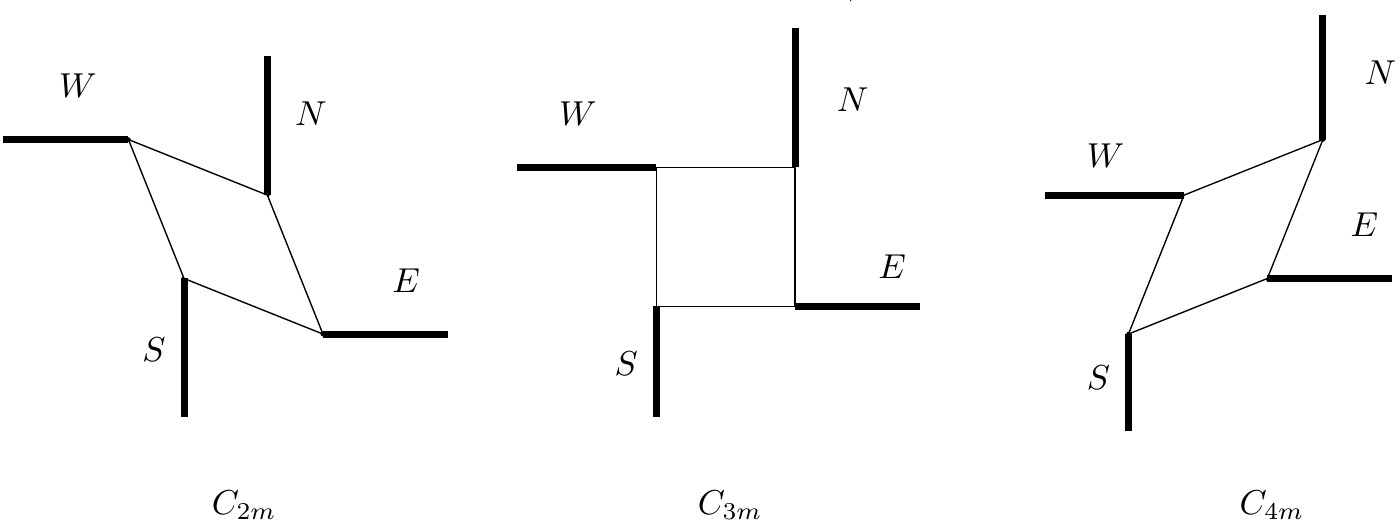}
\caption{Deformation of the rhombus from $C_{2m}$ to $C_{4m}$,
corners suppressed}
\label{fig:p2}
\end{figure}

The first $m$ sets $C_1,\ldots,C_m$ come from the Minkowski sum of $R$ and
the previous horizontal $G_i$: This sum is a hexagon again and the
components $N_i,S_i,E_i,W_i$ are just  extended from the horizontal and
vertical sides of this hexagon the same way as in Theorem~\ref{th:2.07}.
The corners are at the same distance from the horizontal and
vertical components as before. See Figure~\ref{fig:p1}.
The last $C_{5m+1}\ldots,C_{6m}$ sets come from the Minkowski sum of
$R(\pi/2)$ and the corresponding vertical segments $G_i$ analogously.

Note that $N_m,S_m,E_m,W_m$ are halflines, and they remain halflines
in all $C_i$ with $i \in \{m,\ldots,5m\}$. The sets $C_{m+1},\ldots, C_{2m}$
come from gradually rotated copies of $R$.
Then the rhombus $R(\pi/4)$ that defines $C_{2m}$ is
gradually deformed to a square (of side legth $\sqrt 2$) in $C_{3m}$, which
is further deformed to the rhombus $R(-\pi/4)$ in $C_{4m}$, see
Figure~\ref{fig:p2}. Then $R(-\pi/4)$ rotates back to $R(-\pi/2)=R(\pi/2)$
in $C_{5m}$.

The proof that this construction gives $D(T)> 1+\sqrt 2 -\varepsilon$
is based on ideas
similar to those used in Theorem~\ref{th:2.07}: First one shows that an
optimal transversal $T$ has no point in the corners, and second, that $T$
does not visit the same type component $N,S,E,W$ twice. We omit the details.

\medskip

{\bf Acknowledgements.}  IB was supported by ERC Advanced
Research Grant no 267165 (DISCONV) and by National Research,
Development and Innovation Office NKFIH Grants K 111827 and K 116769.
ECs was supported by ERC
grants 306493 and 648017 and by Marie Curie
Fellowship, grant No. 750857.
GyK was partially supported by bilateral research
grant T\'ET 12 MX--1--2013--0006.  GT was supported by the National Research,
Develpoment and Innovation Office NKFIH Grant K-111827.

\bigskip

\noindent
Imre B\'ar\'any \\
Alfr\'ed R\'enyi Institute of Mathematics,\\
Hungarian Academy of Sciences\\
13 Re\'altanoda Street Budapest 1053 Hungary\\
and\\
Department of Mathematics\\
University College London\\
Gower Street, London, WC1E 6BT, UK\\
{\tt barany.imre@renyi.mta.hu}\\

\medskip

\noindent
Endre Cs\'oka \\
Alfr\'ed R\'enyi Institute of Mathematics,\\
Hungarian Academy of Sciences\\
13 Re\'altanoda Street Budapest 1053 Hungary\\
{\tt csokaendre@gmail.com}\\

\medskip
\noindent
Gyula K\'arolyi \\
Alfr\'ed R\'enyi Institute of Mathematics,\\
Hungarian Academy of Sciences\\
13 Re\'altanoda Street Budapest 1053 Hungary\\
and\\
Institute of Mathematics\\
E\"otv\"os University\\
1/C P\'azm\'any P. s\'et\'any Budapest 1117 Hungary\\
{\tt karolyi.gyula@renyi.mta.hu}\\

\medskip
\noindent
G\'eza T\'oth \\
Alfr\'ed R\'enyi Institute of Mathematics,\\
Hungarian Academy of Sciences\\
13 Re\'altanoda Street Budapest 1053 Hungary\\
{\tt toth.geza@renyi.mta.hu}\\

\end{document}